\documentclass[a4paper,12pt]{amsart}
\usepackage{} 
\usepackage{amsmath,amssymb,amsthm}
\usepackage{amscd}
\usepackage{mathrsfs}
\usepackage[usenames,dvipsnames]{color}
\usepackage{pstricks}
\usepackage{graphicx}
\usepackage[all]{xy}
\usepackage{extarrows}
\usepackage{tabularx}
\usepackage{slashed}
\usepackage{tikz}
\usepackage{xparse}
\usepackage{enumitem}
\usepackage[normalem]{ulem}

\usepackage{ifpdf}
\ifpdf
  \usepackage[colorlinks,final,backref=page,hyperindex]{hyperref}
\else
  \usepackage[colorlinks,final,backref=page,hyperindex,hypertex]{hyperref}
\fi

\usepackage{xspace}

\usetikzlibrary{arrows,matrix}
\usetikzlibrary{decorations.pathmorphing}

\def\XXint#1#2#3{{\setbox0=\hbox{$#1{#2#3}{\int}$ }
\vcenter{\hbox{$#2#3$ }}\kern-.58\wd0}}

\def\XXsum#1#2#3{{\setbox0=\hbox{$#1{#2#3}{\sum}$ }
\vcenter{\hbox{$#2#3$ }}\kern-.51\wd0}}

\begin{document}

\newcommand\cutoffint{\mathop{-\hskip -4mm\int}\limits}
\newcommand\cutoffsum{\mathop{-\hskip -4mm\sum}\limits}
\newcommand\cutoffzeta{-\hskip -1.7mm\zeta} 
\newcommand{\goth}[1]{\ensuremath{\mathfrak{#1}}}
\newcommand{\bbox}{\normalsize {}%
        \nolinebreak \hfill $\blacksquare$ \medbreak \par}
\newcommand{\simall}[2]{\underset{#1\rightarrow#2}{\sim}}

\newtheorem{theorem}{Theorem}[section]
\newtheorem{lem}[theorem]{Lemma}
\newtheorem{coro}[theorem]{Corollary}
\newtheorem{problem}[theorem]{Problem}
\newtheorem{conjecture}[theorem]{Conjecture}
\newtheorem{prop}[theorem]{Proposition}
\newtheorem{propdefn}[theorem]{Proposition-Definition}
\newtheorem{lemdefn}[theorem]{Lemma-Definition}
\theoremstyle{definition}
\newtheorem{defn}[theorem]{Definition}
\newtheorem{remark}[theorem]{Remark}
\newtheorem{exam}[theorem]{Example}
\newtheorem{coex}[theorem]{Counterexample}
\newtheorem{algorithm}[theorem]{Algorithm}
\newtheorem{convention}[theorem]{Convention}
\newtheorem{principle}[theorem]{Principle}

\renewcommand{\theenumi}{{\it\roman{enumi}}}
\renewcommand{\theenumii}{{\alph{enumii}}}

\newenvironment{thmenumerate}
{\leavevmode\begin{enumerate}[leftmargin=1.5em]}{\end{enumerate}}

\newcommand{\nc}{\newcommand}
\newcommand{\delete}[1]{}
\newcommand{\aside}[1]{\delete{#1}}

\nc{\mlabel}[1]{\label{#1}}  
\nc{\mcite}[1]{\cite{#1}}  
\nc{\mref}[1]{\ref{#1}}  
\nc{\mbibitem}[1]{\bibitem{#1}} 

\delete{
\nc{\mlabel}[1]{\label{#1}  
{\hfill \hspace{1cm}{\small\tt{{\ }\hfill(#1)}}}}
\nc{\mcite}[1]{\cite{#1}{\small{\tt{{\ }(#1)}}}}  
\nc{\mref}[1]{\ref{#1}{{\tt{{\ }(#1)}}}}  
\nc{\mbibitem}[1]{\bibitem[\bf #1]{#1}} 
}


\newcommand{\bottop}{\top\hspace{-0.8em}\bot}
\nc{\mtop}{\top\hspace{-1mm}}
\nc{\tforall}{\text{ for all }}
\nc{\bfcf}{{\calc}}
\newcommand{\R}{\mathbb{R}}
\newcommand{\FC}{\mathbb{C}}
\newcommand{\K}{\mathbb{K}}
\newcommand{\Z}{\mathbb{Z}}
\nc{\PP}{\mathbb{P}}
\newcommand{\Q}{\mathbb{Q}}
\newcommand{\N}{\mathbb{N}}
\newcommand{\mo}{\mathbb{O}}
\newcommand{\F}{\mathbb{F}}
\newcommand{\T}{\mathbb{T}}
\newcommand{\G}{\mathbb{G}}
\newcommand{\C}{\mathbb{C}}

\newcommand{\bfc}{\calc}

\newcommand {\frakb}{{\mathfrak {b}}}
\newcommand {\frakc}{{\mathfrak {c}}}
\newcommand {\frakd}{{\mathfrak {d}}}
\newcommand {\fraku}{{\mathfrak {u}}}
\newcommand {\fraks}{{\mathfrak {s}}}
\newcommand {\frakP}{{\mathfrak {P}}}

\newcommand {\cala}{{\mathcal {A}}}
\newcommand {\calb}{\mathcal {B}}
\newcommand {\calc}{{\mathcal {C}}}
\newcommand {\cald}{{\mathcal {D}}}
\newcommand {\cale}{{\mathcal {E}}}
\newcommand {\calf}{{\mathcal {F}}}
\newcommand {\calg}{{\mathcal {G}}}
\newcommand {\calh}{\mathcal{H}}
\newcommand {\cali}{\mathcal{I}}
\newcommand {\call}{{\mathcal {L}}}
\newcommand {\calm}{{\mathcal {M}}}
\newcommand {\calp}{{\mathcal {P}}}
\newcommand {\calv}{{\mathcal {V}}}
\newcommand {\calq}{{\mathcal {Q}}}

\newcommand{\conefamilyc}{{\mathfrak{C}}}
\newcommand{\conefamilyd}{{\mathfrak{D}}}

\newcommand{\Hol}{\text{Hol}}
\newcommand{\Mer}{\text{Mer}}
\newcommand{\lin}{\text{lin}}
\nc{\Id}{\mathrm{Id}}
\nc{\ot}{\otimes}
\nc{\bt}{\boxtimes}
\nc{\id}{\mathrm{Id}}
\nc{\Hom}{\mathrm{Hom}}
\nc{\im}{{\mathfrak Im}}

\newcommand{\tddeux}[2]{\begin{picture}(12,5)(0,-1)
\put(3,0){\circle*{2}}
\put(3,0){\line(0,1){5}}
\put(3,5){\circle*{2}}
\put(3,-2){\tiny #1}
\put(3,4){\tiny #2}
\end{picture}}

\newcommand{\tdtroisun}[3]{\begin{picture}(20,12)(-5,-1)
\put(3,0){\circle*{2}}
\put(-0.65,0){$\vee$}
\put(6,7){\circle*{2}}
\put(0,7){\circle*{2}}
\put(5,-2){\tiny #1}
\put(6,5){\tiny #2}
\put(-5,8){\tiny #3}
\end{picture}}

%
%

\nc{\mge}{_{bu}\!\!\!\!{}}

\nc{\vep}{\varepsilon}

\def \e {{\epsilon}}
\nc{\syd}[1]{}

\nc{\prt}{P-}
\nc{\Prt}{P-}
\newcommand{\loc}{locality\xspace}
\newcommand{\Loc}{Locality\xspace}
\nc{\orth}{orthogonal\xspace}
\nc{\tloc}{L-}
\newcommand{\lset}{{\bf \loc SET }}
\newcommand{\set}{{\bf SET }}
\newcommand{\xat}{{X^{_\top 2}}}
\newcommand{\xbt}{{X^{_\top 3}}}
\newcommand{\xct}{{X^{_\top 4}}}

\newcommand{\gat}{{G^{_\top 2}}}
\newcommand{\gbt}{{G^{_\top 3}}}
\newcommand{\gct}{{G^{_\top 4}}}
\nc{\gnt}{{G^{_\top n}}}

\newcommand{\htwot}{{\calh ^{\ot_\top 2}}}
\newcommand{\hbt}{{calh ^{\ot_\top 3}}}
\newcommand{\hct}{{\calh ^{\ot_\top 4}}}

 \nc{\bfk}{{\bf k}}
\nc{\ID}{\mathfrak{I}} \nc{\lbar}[1]{\overline{#1}}
\nc{\bre}{{\rm b}} \nc{\sd}{\cals} \nc{\rb}{\rm RB}
\nc{\A}{\rm angularly decorated\xspace} \nc{\LL}{\rm L}
\nc{\w}{\rm wid} \nc{\arro}[1]{#1}
\nc{\ver}{\rm ver}
\nc{\FL}{F_{\mathrm L}}
\nc{\FNA}{\FN(A)} \nc{\NA}{N_{A}}
\nc{\dr}{\diamond_r}
\nc{\shar}{{\mbox{\cyrs X}}_r} 
\nc{\dt}{\Delta_T}
\nc{\da}{\Delta_A}
\nc{\vt}{\vep_T }
\nc{\bul}{\bullet}
\nc{\free}[1]{\bar{#1}}
\nc{\lt}{{}^\top\!U}
\nc{\rt}{U^\top}
\nc{\lts}{{}^\top\!}
\nc{\weak}{strong\xspace}
\nc{\strong}{strong\xspace}
\nc{\fine}{refined\xspace}


\title[Several locality semigroups, path semigroups and partial semigroups]{Several locality semigroups, path semigroups and partial semigroups}

\author{Shanghua Zheng}
\address{Department of Mathematics, Jiangxi Normal University, Nanchang, Jiangxi 330022, China}
         \email{zhengsh@jxnu.edu.cn}

\date{\today}

\begin{abstract}
Locality semigroups were proposed recently as  one of the basic locality algebraic structures, which are studied in mathematics and physics.
Path semigroups and partial semigroups were also developed by many authors in the literature.
In this paper, we study  free objects in the category of refined locality semigroups. It turns out that the path locality semigroup of a quiver is the free refined locality semigroup on a locality set. We also explore the relationships among locality semigroups, partial semigroups and path locality semigroups, concluding that the path locality semigroup is a proper subclass of the intersection of locality semigroups and partial semigroups. In particular, the class of refined locality semigroups is a proper subclass of strong locality semigroups. Furthermore, we show that, when a partial semigroup is a refined locality semigroup, one can extend it a strong semigroup with zero.
\end{abstract}

\subjclass[2010]{20M05, 18B40, 16G20, 05C38, 08A55}

\keywords{locality semigroup, path semigroup, quiver, partial semigroup, partial algebra, refined locality semigroup}

\maketitle

\tableofcontents

\setcounter{section}{0}

\allowdisplaybreaks

\section{Introduction}
The concept of locality is widely  employed in various branches of mathematics, such as local algebras~\mcite{JS} and local operators in functional analysis~\mcite{BAE,Bat}.   Locality is also used in  computer science and physics, especially in classic and quantum field theory.  For instance, the locality principle is a critical factor in  Einstein's theory.  It is well-known that renormalization, which is a technique to remove the divergences in  Feynman integrals calculations,  plays an important role in quantum field theory~\mcite{Bor,BP,CK,Z} and mathematics~\mcite{GZ,GPZ3,M}.
More recently,  from an algebraic viewpoint, the study about how to preserve locality in the renormalization under the framework of algebraic Birkhoff factorization was proposed by  P. Clavier, L. Guo,  S. Paycha and B. Zhang~\mcite{CGPZ1}. As a result, locality semigroup, locality algebras, locality coalgebras, and locality Rota-Baxters are  established in the context of locality. As a starting point of this paper, we develop the basic results of locality semigroups. One of these results  is  free objects in the category of locality semigroups.

As is well known, a quiver $Q:=(Q_0,Q_1,s,t)$ is  a directed graph and can be viewed as a  basic mathematical object. The theory of representations of quivers  was originally introduced to solve the classification problem of tuples of subspaces of a prescribed vector space  from  linear algebra. Since then quiver representations have been studied quite extensively with board application in mathematics, including  invariant theory, Kac-Moody Lie algebras and quantum groups~\mcite{Br,FH,HJ,Scho}.

Denote $\calp$ by the set of all paths in $Q$. In fact, the definition of multiplication of  path in $Q$ is involved in  partially defined binary  operations, since the composition $pq$ makes sense only if $t(p)=s(q)$ for $p,q\in\calp$.  This means that the multiplication of path is well-defined  only for some elements of $\calp$.  Roughly speaking, the  path  is  analogous to the locality semigroups and possesses some ``freeness" property, and hence is  called a {\bf path locality semigroup} in the sequel. Motivated by the freeness property of path, we explore the basic results of path locality semigroups.   We finally show that the path locality semigroup of a quiver is the free refined locality semigroup on a locality set.  On the other hand, a {\bf path  semigroup} is a path locality semigroup $\calp$ by  adding a zero element (or zero path) and then defining $pq=0$ when $t(p)\neq s(q)$ for $p,q\in\calp$.   In~\mcite{For,MB}, the authors study  formulas for determining effective dimensions of  path semigroups over an uncountable field.

Furthermore, the path algebra $\bfk \calp$ of a quiver $Q$ is essential in the theory of quiver representations as well. For example, a representation of $Q$ is  equivalent to a left module over the path algebra $\bfk \calp$~\mcite{Br}.  Also, according to the well-known Gabriel Theorem~\mcite{ASS,ARS} an algebra over an algebraically closed field is a quotient of the path algebra of its Ext-quiver modulo an admissible ideal.  More generally,  an Artinian algebra over a perfect field is isomorphic to a quotient of the generalized path algebra of its natural quiver~\mcite{LL}. In a recent study ~\mcite{GL} of path algebras,  the authors developed the Lie algebra  of derivations on the path algebra  of a quiver, and gave the characterizations of derivations on a path algebra. More recently, the Hopf algebras on path algebras,  reconstruction of path algebras and the dimensions of path algebras were also studied ~\mcite{AH,HU,KW}. As a consequence, the relationships between path algebras and path locality algebras will be continued in a future work.

A partial semigroup can be regarded as a generalization of a semigroup $(S,\cdot)$ to partial binary operations, that is,  operations are defined only for some elements of $S$.  In~\mcite{Sch}, partial semigroups were also used to develop the coordinatization of all bounded posets.  There are close relationships between the partial semigroups and locality semigroups, because of  the  property that their multiplications are partially defined similarly. Then a natural question comes up: whether or not one of them includes the other. For this reason this paper  discusses  the relationships between locality semigroups and partial semigroups. Meanwhile, we also develop the relationships among \fine locality semigroups, \weak locality semigroups and partial semigroups.

In fact, partial algebras are introduced in order to solve the word problem, which is  whether two words in the generators represent the same element of the algebra ~\mcite{Ev51}. Nowadays, partial algebras are especially useful for theoretical computer science~\mcite{Bur}.
See~\mcite{BBH,Bur, Gr, Gud, LE,SW} for  further details.

In addition, in~\mcite{GS}, the authors give a general way  to construct algebras with given properties,  starting with a simple constructed partial algebra and completing it by using some universal constructions. Thus, the last main goal of this paper is to provide a method to construct a semigroup from any given partial semigroup. We discuss this question in the last section.

The layout of the paper is as follows. In Section~\mref{sec:FreefineLSG} we start by recalling the concept of locality semigroups and give some examples.
We then  give the  construction of free \fine locality semigroups.  In Section~\mref{sec:parSG} we first give the relations among refined locality semigroups, strong locality semigroup and  the intersection between partial semigroups and locality semigroups. We further discuss  in some detail the relationship between partial semigroups and locality semigroups.
Section~\mref{sec:stro} then gives a natural way to construct a strong semigroup with zero from a \fine locality semigroup, and some necessary examples are also provided to show why this method  dose not apply to more general partial semigroups. Finally, we  show that the path locality semigroup is a \fine locality semigroup, and hence is a  strong semigroup with zero.
\smallskip

\noindent

\section{Free \fine locality semigroups}
\mlabel{sec:FreefineLSG}

The main purpose of this section is to construct the free \fine locality semigroup on a locality set. We first recall the basic concepts of locality and give some necessary examples in Section~\mref{subsec:locsg}. We then introduce the path locality semigroups of a quiver in Section~\mref{subsec:pathlocsg}. We show that the path locality semigroup is  the free \fine locality semigroup on a locality set in Section~\mref{subsec:fineLSG}, stated in Theorem~\mref{thm:FreeStrLocSg}.

\subsection{Locality semigroups}
\mlabel{subsec:locsg}

We start by recalling the definition of locality semigroups and extend some related concepts of semigroups, such as subsemigroups and ideals, to that of locality semigroups. Several examples of locality semigroups which often arise in practice are given.

\begin{defn}{\bf(}\cite[Definition 2.1]{CGPZ1}{\bf)}
\begin{enumerate}
\item
A {\bf locality set} is a couple $(X,\top)$, where $X$ is a set and $\top\subseteq X\times X$ is a binary relation on $X$, called a {\bf locality relation} of locality set. When the underlying set $X$ needs to be emphasized,  we also denote $X\times_\top X:=\mtop_X:=\top$.
\item Let $(X,\mtop_X)$ be a locality set. Let $X'$ be a subset of $X$ and let $\mtop_{X'}:=(X'\times X')\cap \mtop_X$. Then the pair $(X',\mtop_{X'})$ is called a {\bf sub-locality set} of $(X,\mtop_X)$.
\item
Let $(X,\top)$ be a locality set. For any subset $U\subseteq X$, let
\begin{equation}
\lt:=\{x\in X\,|\,(x,u)\in X\times_\top X\,\text{for all}\, u\in U\}
\end{equation}
to be the {\bf left  polar subset} of $U$.
Similarly, we let
\begin{equation}
U^\top:=\{x\in X\,|\,(u,x)\in X\times_\top X\,\text{for all}\, u\in U\}
\end{equation}
to be the {\bf right polar subset} of $U$.
\item
Let $(X,\top)$ be a locality set. A map
$$\mu:X\times_{\top} X\to X,\quad (x,y)\mapsto \mu(x,y)\,\, \text{for all}\,\, (x,y)\in \top,$$
is called a {\bf partial binary operation on $X$}. The image $\mu(x,y)$ is written  simply  $x\cdot y$ or $xy$ in the sequel if this ambiguous notation will cause no confusion in context.
\end{enumerate}
\end{defn}
The first example of locality sets and partial binary operations is very simple.
\begin{exam}
Let $\R$ be the set of real numbers. Let $\mtop_\R:=\{(x,y)\in\R^2|\,y\neq 0\}$. Then $(\R,\mtop_\R)$ is a locality set. Take $U:=\{0\}$. Then $U^\top=\R \setminus\{0\}$ and $\lts U=\emptyset$. So $U^\top\neq \lts U$.
We note that the division $\div:\R\times_{\mtop_\R} \R\to \R,\,(x,y)\mapsto \frac{x}{y}$, is a partial binary operation on $\R$.
\end{exam}

\begin{defn}
\begin{enumerate}
\item
A {\bf locality semigroup} is a locality set $(S,\top)$ together with a partial binary operation defined on $S$:
$$\mu_S:S\times_\top S \to S,\quad (a,b)\mapsto a b\,\,\,\text{for all}\, (a,b)\in \top,$$
  such that for all subset $U$ of $S$,
\begin{equation}
\mu_S((\lt\times \lt)\cap \top)\subseteq \lt,
\mlabel{eq:leftclosed}
\end{equation}
and for all subset $U$ of $S$,
\begin{equation}
\mu_S((U^\top\times U^\top)\cap \top)\subseteq U^\mtop,
\mlabel{eq:rightclosed}
\end{equation}
and the {\bf locality associative law }holds: for all $a,b,c\in S$,
\begin{equation}
(a,b),(b,c),(a,c)\in \top\Rightarrow (a b) c=a (b c).
\mlabel{eq:locass}
\end{equation}
We denote a locality semigroup $(S,\top)$ with a partial binary operation $\mu_S$ by $(S,\top,\mu_S)$ or simply $(S,\top)$ if there is no danger of confusion.
\mlabel{it:defnlocsg}
\item We say that a locality semigroup $(S,\top)$ is {\bf transitive} if  $\top$ is transitive, that is, $(a,b),(b,c)\in \top \Rightarrow (a,c)\in \top$.
\item
A {\bf left identity} {\bf[right identity]} of  a locality semigroup $(S,\top)$ is an element $1\in S$  such  that for all $a\in S$,  $(1,a)\in \top$ and $1a=a$ [$(a,1)\in \top$ and $a1=a$].
An {\bf identity} of  $(S,\top)$  is an element $1\in S$ that  is  both a left  and right identity.
A locality semigroup $(S,\top)$ with an identity is called a {\bf locality monoid}, usually denoted by $(S,\top,1)$ or $(M,\top)$.
\item
A {\bf left zero element} {\bf [right zero element]} of  a locality semigroup $(S,\top)$ is an element $0\in S$  such  that for all $a\in S$,  $(0,a)\in \top$ and $0a=0$ [$(a,0)\in \top$ and $a0=0$].
If $0\in S$ is both a left and right zero element, we say that $0$ is a {\bf zero element} of $(S,\top)$, and that $(S,\top)$ is a {\bf locality semigroup with zero}.
\end{enumerate}
\mlabel{defn:locsg}
\end{defn}
\begin{remark} Let $(S,\top)$ be a locality set with a partial binary operation. For any given $a,b,c\in S$, suppose that $(a,b),(a,c),(b,c)\in \top$. On the one hand, if we take $U:=\{c\}$, then $a,b\in \lts U$, and so $(a,b)\in(\lt\times \lt)\cap \top$.  By Eq.~(\mref{eq:leftclosed}), we obtain $(ab,c)\in \top$. On the other hand, if we take $U:=\{a\}$,  then $b,c\in U^\top$, and hence $(a,bc)\in \top$ by Eq.~(\mref{eq:rightclosed}). This shows that both sides of Eq.~(\mref{eq:locass}) make sense.
\end{remark}

\begin{exam}Let $X$ be a nonempty set and let $\mathscr{P}(X)$ be the power set of $X$. Let $\mtop_\mathscr{P}:=\{(A,B)\in \mathscr{P}(X)\times \mathscr{P}(X)\,|\,A\subseteq B\}$. Then $(\mathscr{P}(X),\mtop_\mathscr{P})$ together with the union  operation $\cup$ or the intersection operation $\cap$ is a locality semigroup. Furthermore, $(\mathscr{P}(X),\mtop_\mathscr{P},\cup)$ is a locality semigroup with  left identity  $\emptyset$, and $(\mathscr{P}(X),\mtop_\mathscr{P},\cap)$ is a locality semigroup with left zero element $\emptyset$. Both are also transitive, since the inclusion relation $\subseteq$ on $\mathscr{P}(X)$ is a partial order.
\mlabel{exam:powset}
\end{exam}
\begin{exam}{\bf(}\cite[Example 3.6]{CGPZ1}{\bf)}
Let $\N^+$ be the set of positive integers and let $a,b\in \N^+$. Let $\gcd(a,b)$ be the greatest common divisor of $a,b$. Denote
$$\mtop_{cop}:=\{(a,b)\,|\, \gcd(a,b)=1,a,b\in \N^+\}.$$
Since $\gcd(a,b)=\gcd(b,a)$, we have $\lt=U^\top$ for all $U\subseteq \N^+$.  Let $u\in U$. If $\gcd(u,a)=1$ and $\gcd(u,b)=1$, then $\gcd(u,ab)=1$, and so Eqs.~(\mref{eq:leftclosed}) and (\mref{eq:rightclosed}) hold. Then $(\N^+,\mtop_{cop})$ with the ordinary multiplication on $\N^+$ is a locality monoid.
\mlabel{exam:PosIn}
\end{exam}

\begin{exam}
Let $\N$ be the set of nonnegative integers. Denote
$$\top:=\mtop_{cop}^{\,0}:=\mtop_{cop}\cup\{(0,0),(0,a),(a,0)\,|\, a\in \N^+\}.$$
Then $\{0\}^\mtop={}^\top\!\{0\}=\N$, and  if  $(0,a)\in \mtop_{cop}^{\,0}$ and $(0,b)\in \mtop_{cop}^{\,0}$ with $(a,b)\in \mtop_{cop}^{\,0}$ for all $a,b\in \N$, then $(0,a b)\in \mtop_{cop}^{\,0}$. Thus, Eqs.~(\mref{eq:leftclosed}) and (\mref{eq:rightclosed}) hold by Example~\mref{exam:PosIn},
and hence $(\N,\mtop_{cop}^{\,0})$ with the ordinary multiplication on $\N$ is a locality semigroup with zero.
\end{exam}
As a result, we obtain a way, analogous to that for semigroups~\mcite{How,PRe}, to construct a locality monoid or locality semigroup with zero from a locality semigroup as follows.
\begin{lem}Let $(S,\mtop_S)$ be a locality semigroup. Let $0,1\notin S$.
\begin{enumerate}
\item
Let $S^1:=S\cup\{1\}$. Denote $\mtop_{S^1}:=\mtop_S\cup\{(1,1),(1,a),(a,1)\,|\, a\in S\}$. Define a partial binary operation on $S^1$:
$$\mu_1:S^1\times_{\mtop_{S^1}}S^1\to S^1,\quad \mu_1(a,b):=\left\{\begin{array}{llll}ab,\quad\text{if}\,\,(a,b)\in \mtop_S;\\
b,\quad\text{if}\,\, a=1;\\
a,\quad\text{if}\,\, b=1.\end{array}\right.$$
Then $(S^1,\mu_1)$ is a locality monoid.
\item
 Let $S^0:=S\cup\{0\}$. Denote $\mtop_{S^0}:=\mtop_S\cup\{(0,0),(0,a),(a,0)\,|\, a\in S\}$.  Define a partial binary operation on $S^0$:
$$\mu_0:S^0\times_{\mtop_{S^0}}S^0\to S^0,\quad \mu_0(a,b):=\left\{\begin{array}{llll}ab,\quad\text{if}\,\,(a,b)\in \mtop_S;\\
0,\quad\text{if}\,\, a=0\,\text{or}\,\, b=0.\end{array}\right.$$
Then $(S^0,\mu_0)$ is a locality semigroup with zero.
\end{enumerate}
\end{lem}
\begin{defn}Let $(S,\mtop_S)$ be a locality semigroup. Let $\emptyset\neq A\subseteq S$ and let $\mtop_A:=(A\times A) \cap \mtop_S$.
\begin{enumerate}
\item
A sub-locality set $(A,\mtop_A)$ is called a {\bf sub-locality semigroup} of $(S,\mtop_S)$ if it is closed under the partial binary operation of $(S,\mtop_S)$, i.e. it satisfies the condition:
for all $(a,b)\in \mtop_A$, $ab\in A$;
\item
Let $\mtop_{\ell}:=(S\times A) \cap \mtop_S$. The pair $(A,\mtop_\ell)$ is called a {\bf left locality ideal}  of $(S,\mtop_S)$ if it satisfies the condition: for all $(s,a)\in \mtop_\ell$, $sa\in A$;
\item
Let $\mtop_r:=(A\times S) \cap \mtop_S$. The pair $(A,\mtop_r)$ is called a {\bf right locality ideal} of $(S,\mtop_S)$ if it satisfies the condition: for all $(a,s)\in \mtop_r$, $as\in A$;
\item
The pair $(A,\mtop_\ell\cup\mtop_r)$ is called a {\bf locality ideal} of  $(S,\mtop_S)$  if  $(A,\mtop_\ell)$ is a left locality ideal and $(A,\mtop_r)$ is a right locality ideal, that is, for all $s\in S$ and all $a\in A$, if $(s,a)\in \mtop_\ell$ and $(a,s)\in \mtop_r$, then $sa\in A$ and $as\in A$.
\end{enumerate}
\end{defn}
We know that if  $(A,\mtop_\ell\cup\mtop_r)$ is a  locality ideal, then $(A,\mtop_A)$ is a locality semigroup,  but not vice versa. For example,
\begin{exam}
\begin{enumerate}
\item
According to Example~{\mref{exam:PosIn}}, we let $\mo: =2\N+1$ be the set of positive odd integers, and let $\mtop_\mo:=(\mo\times \mo)\cap \mtop_{cop}$. Then $(\mo,\mtop_\mo)$ is a sub-locality semigroup of $(\N^+,\mtop_{cop})$, since the multiplication of two odd numbers is still odd.  But if we let $\mtop_{\ell}:=(\N^+\times \mo) \cap \mtop_{cop}$, then $(\mo,\mtop_\ell)$ is not a left locality ideal of $(\N^+,\mtop_{cop})$ (although $(2,3)\in\mtop_{\ell}$,  $6\notin \mo$). Also, let $\mtop_{r}:=(\mo\times \N^+) \cap \mtop_{cop}$. Then $(\mo,\mtop_r)$ is also not a right locality ideal. Thus, $(\mo,\mtop_\ell \cup\mtop_r)$ is not a locality ideal of $(\N^+,\mtop_{cop})$.
\item
By Example~{\mref{exam:powset}}, we choose a fixed $x_0\in X$ and let $\mathscr{P}_{x_0}:=\{A\in \mathscr{P}(X)\,|\,x_0\in A\}$. Denote $\mtop_{\mathscr{P}_{x_0}}:=(\mathscr{P}_{x_0}\times\mathscr{P}_{x_0})\cap \mtop_\mathscr{P}$. Then we see that $(\mathscr{P}_{x_0},\mtop_{\mathscr{P}_{x_0}})$ is a sub-locality semigroup  of $(\mathscr{P}(X),\mtop_\mathscr{P},\cup)$. In addition, if we let $\mtop_\ell:=(\mathscr{P}(X)\times \mathscr{P}_{x_0})\cap \mtop_\mathscr{P}$, then $(\mathscr{P}_{x_0},\mtop_\ell)$ is a left locality ideal of $(\mathscr{P}(X),\mtop_\mathscr{P},\cup)$. Let $\mtop_r:=(\mathscr{P}_{x_0}\times \mathscr{P}(X))\cap \mtop_\mathscr{P}$. Then $(\mathscr{P}_{x_0},\mtop_r)$ is a right locality ideal of $(\mathscr{P}(X),\mtop_\mathscr{P},\cup)$. Thus, $(\mathscr{P}_{x_0},\mtop_\ell\cup\mtop_r)$ is a locality ideal.
\end{enumerate}
\end{exam}

\begin{remark}
For every sub-locality set $(A,\mtop_A)$ of $(S,\mtop_S)$, there is at least sub-locality semigroup containing  $(A,\mtop_A)$. Denote by $\{ (A_i,\mtop_{A_i})\,|\, i\geq 1\}$, where $(A_i,\mtop_{A_i})$ is a sub-locality semigroup of $(S,\mtop_S)$ containing  $(A,\mtop_A)$ for each $i\geq 1$, the set consisting of  all sub-locality semigroups containing $(A,\mtop_A)$.  Note that
\begin{equation}
\cap_{i\geq 1}\mtop_{A_i}=\big((\cap_{i\geq 1}A_i)\times (\cap_{i\geq 1}A_i)\big)\cap \mtop_S.
\mlabel{eq:Inter}
\end{equation}
Let
$$\bigcap_{i\geq 1}(A_i,\mtop_i):=(\cap_{i\geq 1}A_i, \cap_{i\geq 1}\mtop_{A_i})$$
be the intersection of all sub-locality semigroups containing $(A,\mtop_A)$. Then by Eq.~(\mref{eq:Inter}), $\bigcap_{i\geq 1}(A_i,\mtop_i)$ is a sub-locality semigroup of $(S,\mtop_S)$. We denote it by $\langle A,\top_A \rangle$, and we  call  $\langle A,\top_A \rangle$ the {\bf sub-locality semigroup  of $(S,\mtop_S)$ generated by the sub-locality set $(A,\mtop_A)$}, i.e. the smallest sub-locality semigroup containing $(A,\mtop_A)$.
\end{remark}

\begin{defn}
Let $(X,\mtop_X)$ and $(Y,\mtop_Y)$ be locality sets. A set map $\phi:X\to Y$ is called a {\bf locality map} if it satisfies $(\phi \times \phi)(\mtop_X)\subseteq \mtop_Y$, that is, $(\phi\times \phi)(x_1,x_2):=(\phi(x_1),\phi(x_2))\in \mtop_Y$ for all $(x_1,x_2)\in \mtop_X$.
\end{defn}

\begin{defn}
Let $(S_1,\mtop_{S_1},\cdot_{S_1})$ and $(S_2,\mtop_{S_2},\cdot_{S_2})$ be locality semigroups. A set map $\phi:S_1\to S_2$ is called a {\bf locality semigroup homomorphism} if it satisfies the following conditions:
\begin{enumerate}
\item
$\phi$ is a locality map;
\item
$\phi$ is {\bf locality multiplicative}: for all $(a,b)\in \mtop_{S_1}$, we have $\phi(a\cdot_{S_1} b)=\phi (a)\cdot_{S_2}\phi(b)$.
\end{enumerate}
\end{defn}
Furthermore, let $(S_1,\mtop_{S_1},1_{S_1})$ and $(S_2,\mtop_{S_2}, 1_{S_2}) $ be monoids. We say that a locality semigroup homomorphism $\phi: (S_1,\mtop_{S_1},1_{S_1})\to (S_2,\mtop_{S_2},1_{S_2}$) is a {\bf locality monoid homomorphism} if $\phi (1_{S_1})=1_{S_2}$.
For instance, $(\N^+,\mtop_{cop})$ is a locality monoid according to Example~\mref{exam:PosIn}~\mcite{CGPZ1}. Then the Euler's totient function $\varphi:\N^+\to \N^+,\,n\mapsto \varphi(n)$, counting the positive integers coprime to (but not bigger than) $n$, is a locality monoid homomorphism from $(\N^+,\mtop_{cop})$ to $(\N^+,\mtop_{ful})$, where $\mtop_{ful}$ is the full relation on $\N^+$.

\subsection{Path locality semigroups of a quiver}
\mlabel{subsec:pathlocsg}
Now we introduce the path locality semigroups of a quiver.
\begin{defn}
A {\bf quiver}  is a quadruple $ Q:=(Q_0,Q_1,s,t)$, where
\begin{enumerate}
\item
$Q_0$ is a  set, called the {\bf vertex set};
\item
$Q_1$ is also a  set, called the {\bf arrow set};
\item
$s:Q_1\to Q_0$ is a map, called the {\bf source function}, and $t:Q_1\to Q_0$ is a map,  called the {\bf target function}.
\end{enumerate}
\end{defn}
We shall denote the vertices $x,y,z,\cdots$ in the vertex set $Q_0$ and denote the arrows $\alpha, \beta,\gamma,\cdots$ in the arrow set $Q_1$. For every arrow $\alpha \in Q_1$, if $s(\alpha)=x$ [resp. $t(\alpha)=y$], then we call $x$ [resp. $y$] a {\bf source} [resp. {\bf target}] of $\alpha$. An arrow with a source $x$ and a target $y$ will be denoted by $\alpha:x\to y$, or simply by $x\overset{\alpha}{\to} y$. Thus, for example, the quiver with vertices $x,y$ and arrows $\alpha:x\to y$, $\beta_1: x\to x$ and $\beta_2: y\to y$, can be depicted  as follows:
\begin{equation}
 \xymatrix{
    x \ar@(ul,dl)[]_{\beta_1} \ar[r]^{\alpha} & y \ar@(ur,dr)[]^{\beta_2}}.
\mlabel{eq:quiver}
\end{equation}

\begin{defn}
Let $Q:=(Q_0,Q_1,s,t)$ be a quiver and let $k\geq 1$.
\begin{enumerate}
\item
A {\bf path} in $Q$ is either a vertex $v\in Q_0$, usually called an {\bf empty path} or a {\bf trivial path} and often denoted by $e_v$, or a sequence $p:=\alpha_1\alpha_2\cdots\alpha_k$ of arrows, where $\alpha_i\in Q_1$ for each $1\leq i\leq k$, and  $t(\alpha_i)=s(\alpha_{i+1})$ for all $1\leq i\leq k-1$.
\item
Let $k\geq 1$. For all nonempty path $p=\alpha_1\alpha_2\cdots\alpha_k$ in $Q$, where $\alpha_i\in Q_1$ for each $i=1,\cdots,k$, we call $s(p):=s(\alpha_1)$ the {\bf source} of $p$ and call $t(p):=t(\alpha_k)$ the {\bf target} of $p$. In that case we say that the {\bf length} of $p$ is $k$, denoted by $\ell(p)$. By convention, if $p=e_v$ is an empty path, we say that $s(p)=t(p)=v$, and the length of $p$ is $0$. Then $\ell(p)=0$ if and only if $p$ is an empty path.
\item
An {\bf oriented cycle} is a path $p$ with $s(p)=t(p)$.
\end{enumerate}
\end{defn}
When there is no danger of confusion, we also denote $Q:=(Q_0,Q_1)$.
\begin{remark}
We denote by $\calp$ the set of paths in a quiver $Q$.
By the definition of length of a path in $Q$, we identity $Q_0$ and $Q_1$ with the set of all paths of length $0$ and the set of all paths of length $1$, respectively.
More generally, for all $n\geq 0$, we define
\begin{equation}
Q_n:=\{p\in \calp\,|\, \ell(p)=n\}.
\end{equation}
Thus we get
\begin{equation}
\calp=\bigsqcup_{n\geq 0} Q_n,
\end{equation}
the disjoint union of the sets $Q_n$.
\end{remark}

\begin{exam}
Let $Q$ be the quiver
$$\xymatrix{
&x\ar[r]^{\alpha} &y,  \,\,y \ar[r]^{\beta} &z.}
$$
Then $\calp=\{e_x,e_y,e_z,\alpha,\beta,\alpha\beta\}$, where
$$\xymatrix{
&\alpha\beta:=x\ar[r]^{\alpha} &y \ar[r]^{\beta} &z,}$$ the composition of  paths $\alpha$ and $\beta$.
\end{exam}
Let $Q$ be a quiver.
Let $\calp$ be the set of all paths in $Q$.  Then we define
\begin{equation}
\mtop_\calp:=\calp \times_\top \calp:=\{(p,q)\,|\, t(p)=s(q),p,q\in \calp\}.
\mlabel{eq:mtopcalp}
\end{equation}
For all $(p,q)\in \mtop_\calp$, we define $pq$ to be the composition of  paths $p$ and $q$.
\begin{defn}
Let $Q$ be a quiver and
let $\calp$ be the set of all paths in $Q$. Let $p\in \calp$.
\begin{enumerate}
\item
The expression
$$p=v_0\alpha_1v_1\alpha_2\cdots v_{k-1}\alpha_k v_k,$$
where $v_i\in Q_0$ for $0\leq i\leq k$ and $\alpha_j\in Q_1$, and $s(\alpha_j)=v_{j-1}$ and $t(\alpha_j)=v_j$ for $1\leq j\leq k$, is called the {\bf standard decomposition of $p$}.
\item
The expression
$$p:=\alpha_1\alpha_2\cdots\alpha_k,$$
where $\alpha_i\in Q_1$ for $1\leq i\leq k$, is called the {\bf standard decomposition of $p$ into arrows}.
\end{enumerate}
\end{defn}
For every $p\in \calp$, we know that the standard decomposition of $p$ into arrows  is unique.
Here is a natural way to construct a locality semigroup  from the set of paths $\calp$.
\begin{prop}
Let $Q$ be a quiver and
let $\calp$ be the set of all paths in $Q$.  Denote
$$\mtop_\calp:=\calp \times_\top \calp:=\{(p,q)\,|\, t(p)=s(q),p,q\in \calp\}.$$
A partial binary operation $\mu_\calp:\calp\times_\top\calp\to \calp$ is  given by $(p,q)\mapsto pq$,  the composition of  paths $p$ and $q$. Then $(\calp,\mtop_\calp,\mu_\calp)$  is a locality semigroup, called a {\bf path locality semigroup} of $Q$.
\mlabel{prop:pathsg}
\end{prop}
\begin{proof}
We shall prove that $\mu_\calp((U^{\mtop_\calp}\times U^{\mtop_\calp})\cap \mtop_\calp)\subseteq U^{\mtop_\calp}$ and $\mu_\calp(({}^{\mtop_\calp}\!U\times{}^{\mtop_\calp}\!U)\cap \mtop_\calp)\subseteq {}^{\mtop_\calp}\!U$ for all subset $U$ of $\calp$.
Let $(p,q)\in (U^{\mtop_\calp}\times U^{\mtop_\calp})\cap \mtop_\calp$. This means that  $(u,p)\in \mtop_\calp$ and $(u,q)\in \mtop_\calp$ for all $u\in U$, and so $t(u)=s(p)=s(q)$. Then  $s(pq)=s(p)=t(u)$, and hence $(u,pq)\in \mtop_\calp$ holds for all $u\in U$. Thus $pq\in U^{\mtop_\calp}$. On the other hand, let $(p,q)\in ({}^{\mtop_\calp}\!U\times{}^{\mtop_\calp}\!U)\cap \mtop_\calp$. Then $(p,u)\in \mtop_\calp$ and $(q,u)\in \mtop_\calp$ for all $u\in U$, and so $t(p)=s(u)=t(q)$. Thus $t(pq)=t(q)=s(u)$. Then we get $(pq,u)\in \mtop_\calp$ for all $u\in U$. This gives $pq\in {}^{\mtop_\calp}\!U$.

Next we verify that the locality associative law holds, that is,
$(p_1 p_2) p_3=p_1 (p_2 p_3) $ for all $(p_1,p_2),(p_1, p_3),(p_2,p_3)\in \mtop_\calp$.
Let $(p_1,p_2),(p_1, p_3),(p_2,p_3)\in \mtop_\calp$. Then $t(p_1)=s(p_2)$ and $t(p_2)=s(p_3)$. Thus $t(p_1p_2)=t(p_2)=s(p_3)$ and $t(p_1)=s(p_2)=s(p_2p_3)$, and hence $(p_1p_2,p_3)\in \mtop_\calp$ and $(p_1,p_2p_3)\in \mtop_\calp$. This shows that
$$(p_1 p_2) p_3=p_1p_2p_3=p_1(p_2 p_3).$$
\end{proof}

\subsection{Free \fine locality semigroups on a locality set}
\mlabel{subsec:fineLSG}

We then give a explicit construction of free objects in the category of \fine locality semigroups. We begin by introducing the definition of \fine locality semigroups.
\begin{defn}
Let $(S,\top)$ be a locality set. A {\bf \fine locality semigroup} is a locality set $(S,\top)$ together with a partial binary operation:
$$\mu_S:S\times_\top S \to S,\quad (a,b)\mapsto ab\,\,\,\text{for all}\, (a,b)\in \top,$$
such that for all $a,b,c\in S$,
\begin{enumerate}
\item If $(a,b)\in\top$, then $(b,c)\in\top$ if and only if $(ab,c)\in\top$; and
\mlabel{it:fineone}
\item If $(b,c)\in\top$, then $(a,b)\in \top$ if and only if $(a,bc)\in\top$; and
\mlabel{it:finetwo}
\item
If $(a,b),(b,c)\in \top $, then
$(ab)c=a(bc).$
\mlabel{it:fineass}
\end{enumerate}
\mlabel{defn:flocsg}
\end{defn}
\begin{remark}
By Definition~\mref{defn:flocsg}(\mref{it:fineone}) and (\mref{it:finetwo}), we have
\begin{equation}
(a,b),(b,c)\in\top\Rightarrow (ab,c),(a,bc)\in\top\quad \text{for all}\,\,a,b,c\in S.
\mlabel{eq:refinedremark}
\end{equation}
Thus $(ab)c$ and $a(bc)$ make sense.
\end{remark}
\begin{prop}Let $(\calp,\mtop_\calp,\mu_\calp)$ be the path locality semigroup as defined in Proposition~\mref{prop:pathsg}. Then
 $(\calp,\mtop_\calp,\mu_\calp)$ is also a refined locality semigroup.
 \mlabel{prop:pathresg}
\end{prop}
\begin{proof}
We first show that $(\calp,\mtop_\calp)$ satisfies the condition (\mref{it:fineone}) of  Definition~\mref{defn:flocsg}. For all $p_1,p_2,p_3\in\calp$, assume $(p_1,p_2)\in \mtop_\calp$. If $(p_2,p_3)\in\mtop_\calp$, then $t(p_2)=s(p_3)$, and so $t(p_1p_2)=t(p_2)=s(p_3)$. Thus, $(p_1p_2,p_3)\in\mtop_\calp$. For the converse, if $(p_1p_2,p_3)\in\mtop_\calp$, then $t(p_2)=t(p_1p_2)=s(p_3)$, and hence $(p_2,p_3)\in\mtop_\calp$. Secondly, we  verify that the condition (\mref{it:finetwo}) of Definition~\mref{defn:flocsg} holds. Suppose $(p_2,p_3)\in\mtop_\calp$. If $(p_1,p_2)\in\mtop_\calp$, then $t(p_1)=s(p_2)=s(p_2p_3)
$, and thus $(p_1,p_2p_3)\in\mtop_\calp$. Conversely, if $(p_1,p_2p_3)\in\mtop_\calp$, then $t(p_1)=s(p_2p_3)=s(p_2)$, and so $(p_1,p_2)\in\mtop_\calp$. Finally, we prove that  $(\calp,\mtop_\calp)$ satisfies the condition~(\mref{it:fineass}) of Definition~\mref{defn:flocsg}. Suppose $(p_1,p_2),(p_2,p_3)\in\mtop_\calp$. Then $t(p_1)=s(p_2)$ and $t(p_2)=s(p_3)$, and this gives equations $t(p_1)=s(p_2p_3)$ and $t(p_1p_2)=s(p_3)$. Hence $(p_1,p_2p_3),(p_1p_2,p_3)\in\mtop_\calp$, and thus $(p_1p_2)p_3=p_1(p_2p_3)$.
\end{proof}

\begin{lem}
Let $(S,\mtop_S)$ be a  \fine locality semigroup and let $n\geq 2$.  Let $(X,\mtop_X)$ be a locality set and let $(x_i,x_{i+1})\in \mtop_X$ for $i=1,\cdots,n-1$. If $f:(X,\mtop_X)\to (S,\mtop_S)$ is a locality map, then
\begin{equation}
(f(x_1)\cdots f(x_{n-1}),f(x_{n}))\in \mtop_S
\mlabel{eq:fs1}
\end{equation}
 and
\begin{equation}
(f(x_1),f(x_2)\cdots f(x_{n}))\in \mtop_S.
\mlabel{eq:fs2}
\end{equation}
\mlabel{lem:fs}
\end{lem}
\begin{proof}
We prove Eqs.~(\mref{eq:fs1}) and (\mref{eq:fs2}) by induction on $n\geq 2$. For $n=2$, we have $(f(x_1),f(x_2))\in \mtop_S$ by $f$ being a locality map.
Assume that Eqs.~(\mref{eq:fs1}) and (\mref{eq:fs2}) have been proved for $n\leq k$. Consider $n=k+1$. By hypothesis, $(x_i,x_{i+1})\in \mtop_X$ for $i=1,\cdots,k$ . Since $f$ is a locality map, $(f(x_{i}),f(x_{i+1}))\in \mtop_S$ for all $i$, and especially $(f(x_k),f(x_{k+1}))\in\mtop_S$. By the induction hypothesis, we get
$$(f(x_1)\cdots f(x_{k-1}), f(x_{k}))\in \mtop_S,\quad
(f(x_1),f(x_2)\cdots f(x_{k}))\in \mtop_S$$
 and $ (f(x_2)\cdots f(x_{k}),f(x_{k+1}))\in \mtop_S$.
 Since $(S,\mtop_S)$ is a \fine locality semigroup,  we have
 $$(f(x_1)\cdots f(x_{k-1})f(x_k),f(x_{k+1}))\in \mtop_S$$
 and
 $$(f(x_1),f(x_2)\cdots f(x_{k})f(x_{k+1}))\in \mtop_S.$$
This completes the induction, and thus the proof.
\end{proof}
\begin{lem}
Let $(S,\mtop_S)$ be a \fine locality semigroup and let $m,n\geq 1$.  Let $(X,\mtop_X)$ be a locality set. Let $(x_i,x_{i+1})$, $(y_j,y_{j+1})\in \mtop_X$ for $i=1,\cdots,m-1$ and $j=1,\cdots,n-1$.  If $f:(X,\mtop_X)\to (S,\mtop_S)$  is a locality map and $(x_m,y_1)\in \mtop_X$, then
\begin{equation}
(f(x_1)\cdots f(x_m),f(y_1)\cdots f(y_n))\in \mtop_S.
\mlabel{eq:fmul}
\end{equation}
\mlabel{lem:fmul}
\end{lem}
\begin{proof} We first verify that Eq.~(\mref{eq:fmul}) holds for special cases $m=1$ or $n=1$.  If $m=1$, then by hypothesis,
$$(x_1,y_1), (y_1,y_2),\cdots,(y_{n-1},y_n)\in \mtop_X.$$
 By Eq.~(\mref{eq:fs2}), $(f(x_1),f(y_1)\cdots f(y_n))\in \mtop_S$. Similarly, if $n=1$, then
$$(x_1,x_2), \cdots,(x_{m-1},x_m),(x_m,y_1)\in \mtop_X,$$
and so $(f(x_1)\cdots f(x_m),f(y_1))\in \mtop_S$ by Eq.~(\mref{eq:fs1}).

Next we consider $m,n\geq 2$.  By Lemma~\mref{lem:fs}, we obtain
\begin{equation}
(f(x_1)\cdots f(x_{m-1}), f(x_m))\in \mtop_S
\end{equation}
and
\begin{equation}
(f(y_1),f(y_2)\cdots f(y_n))\in \mtop_S.
\mlabel{eq:fbeta}
\end{equation}
Since $f$ is a locality map, we have $(f(x_m),f(y_1))\in \mtop_S$. By $(S,\mtop_S)$ being a \fine locality semigroup and Eq.~(\mref{eq:refinedremark}), we get
$$(f(x_1)\cdots f(x_{m-1})f(x_m),f(y_1))\in \mtop_S.$$
Thus, together with Eqs.~(\mref{eq:refinedremark}) and (\mref{eq:fbeta}),
$$(f(x_1)\cdots f(x_{m-1})f(x_m),f(y_1)f(y_2)\cdots f(y_n))\in \mtop_S.$$
\end{proof}
\begin{defn}
{\rm
Let $(X,\mtop_X)$ be a locality set. A {\bf free \fine locality semigroup} on $(X,\mtop_X)$ is a \fine locality semigroup $(F_L(X),\mtop_F)$ together with a locality  map $j_X:
(X,\mtop_X)\to (\FL(X),\mtop_F)$ such that, for any \fine locality semigroup $(S,\mtop_S)$ and
any locality map $f:(X,\mtop_X)\to (S, \mtop_S)$, there exists a unique
locality semigroup homomorphism $\free{f}: (\FL(X),\mtop_F)\to (S,\mtop_S)$
such that $f=\free{f}\circ j_X$, that is, the following diagram
$$\xymatrix{ (X,\mtop_X)  \ar[rr]^{j_X}\ar[drr]_{f} && (\FL(X),\mtop_F) \ar[d]^{\free{f}} \\
&& (S,\mtop_S)}
$$
commutes.}
\mlabel{defn:freelocsg}
\end{defn}

We next give a construction of free \fine locality semigroups on a locality set.
Let $Q$ be a quiver.
Denote
\begin{equation}
\mtop_\calq:=Q\times_\mtop Q:=\{(\alpha,\beta)\,|\,t(\alpha)=s(\beta),\alpha,\beta\in Q\}.
\end{equation}
Then $(Q,\mtop_\calq)$ is a locality set.
Define a set map
\begin{equation}
j_\calq: Q\to \calp,\quad\alpha\mapsto \alpha,\quad \alpha\in Q.
\end{equation}
Then $(j_\calq\times j_\calq)(\alpha,\beta)=(\alpha,\beta)\in \mtop_\calp$ for all $(\alpha,\beta)\in \mtop_\calq$, and so $j_\calq$ is a locality map. We next give the  main result of this section.
\begin{theorem}
With notations as above, the path locality semigroup $(\calp,\mtop_\calp, \mu_\calp)$ with  $j_\calq$ is a free \fine locality semigroup on a locality set $(Q,\mtop_\calq)$.
\mlabel{thm:FreeStrLocSg}
\end{theorem}
\begin{proof}
By Proposition~\mref{prop:pathresg}, the path locality semigroup $(\calp,\mtop_\calp, \mu_\calp)$ is a \fine locality semigroup. We next show that $(\calp,\mtop_\calp, \mu_\calp)$ satisfies the required universal property.
Let $(S,\mtop_S)$ be a \fine locality semigroup and let $f: (Q,\mtop_\calq) \to (S,\mtop_S)$ be a locality map. Define a set map
$$
\bar{f}: \calp \to S,\quad
p\mapsto  \bar{f}(p):=f(\alpha_1)f(\alpha_2)\cdots f(\alpha_k),
$$
where $ p=\alpha_1\cdots\alpha_k\in \calp$,  and $\alpha_i\in Q$ for $1\leq i\leq k.$  Since $(\alpha_i,\alpha_{i+1})\in\mtop_\calq$ for $i=1,\cdots,k-1$, we have $(f(\alpha_1),f(\alpha_2)\cdots f(\alpha_k))\in\mtop_S$ by Eq.~(\mref{eq:fs2}) in Lemma~\mref{lem:fs}, and thus $\bar{f}$ is well-defined. Firstly, we see that $f=\bar{f}\circ j_\calq$.
We then prove that $\bar{f}$ is indeed a  locality semigroup homomorphism, that is, $(\bar{f}\times\bar{f})(p_1,p_2)=(\bar{f}(p_1),\bar{f}(p_2))\in \mtop_S$  and $\bar{f}(p_1p_2)=\bar{f}(p_1)\bar{f}(p_2)$ for all $(p_1,p_2)\in \mtop_\calp$. Let $(p_1,p_2)\in \mtop_\calp$.
Suppose that $p_1=\alpha_1\cdots\alpha_m$ and $p_2=\beta_1\cdots \beta_n$ with $m,n\geq 1$, where $\alpha_i, \beta_j\in Q$ for $1\leq i\leq m $ and $1\leq j\leq n$. Then we obtain $(\alpha_i,\alpha_{i+1})\in \mtop_\calq$ and  $(\beta_j,\beta_{j+1})\in \mtop_\calq$ for $i=1,\cdots, m-1$ and $j=1,\cdots,n-1$. By $(p_1,p_2)\in \mtop_\calp$,
$$t(\alpha_m)=t(p_1)=s(p_2)=s(\beta_1),$$
and so $(\alpha_m,\beta_1)\in \mtop_\calp$.
Thus, by Lemma~\mref{lem:fmul}, we get
$$(f(\alpha_1)\cdots f(\alpha_{m-1})f(\alpha_m),f(\beta_1)f(\beta_2)\cdots f(\beta_n))\in \mtop_S.$$
This means that $(\bar{f}(p_1),\bar{f}(p_2))\in \mtop_S$. Now we prove that $\bar{f}(p_1p_2)=\bar{f}(p_1)\bar{f}(p_2)$.
By the definition of $\bar{f}$, we have
\begin{eqnarray*}
\bar{f}(p_1p_2)&=&\bar{f}(\alpha_1\cdots\alpha_m\beta_1\cdots\beta_n)\\
&=&f(\alpha_1)\cdots f(\alpha_m)f(\beta_1)\cdots f(\beta_n)\\
&=&\bar{f}(p_1)\bar{f}(p_2).
\end{eqnarray*}
To complete the proof, we finally verify the uniqueness of $\bar{f}$. Assume that there is another  locality semigroup homomorphism $\tilde{f}:(\calp,\mtop_\calp)\to (S,\mtop_S)$ such that $f=\tilde{f}\circ j_\calq$. For every $p=\alpha_1\cdots\alpha_k\in \calp$, we have
\begin{eqnarray*}
\tilde{f}(p)&=&\tilde{f}(\alpha_1\cdots \alpha_k)\\
&=&\tilde{f}(\alpha_1)\cdots\tilde{f}(\alpha_k)\\
&=&\tilde{f}(j_\calq(\alpha_1))\cdots\tilde{f}(j_\calq(\alpha_k))\\
&=&f(\alpha_1)\cdots f(\alpha_k)\\
&=&\bar{f}(p).
\end{eqnarray*}
Thus $\tilde{f}=\bar{f}$, as desired.
\end{proof}

\section{The relationships among several classes locality semigroups}
\mlabel{sec:parSG}
In this section, we  first introduce the \weak locality semigroups and partial semigroups.  We then investigate the relationships among \weak locality semigroups, partial semigroups and \fine locality semigroups.
\subsection{Strong locality semigroups}
\mlabel{subsec:strlsg}
We first introduce the concept of \weak locality semigroups.
\begin{defn}Let $(S,\top)$ be a locality set. A {\bf \weak locality semigroup} is a locality set $(S,\top)$ together with a partial binary operation defined on $S$:
$$\mu_S:S\times_\top S \to S,\quad (a,b)\mapsto ab\,\,\text{for all}\, (a,b)\in \top,$$
such that the {\bf \weak locality  associative law} holds: for all $a,b,c\in S$,
\begin{equation}
(a,b), (b,c)\in \top \Rightarrow (ab,c),(a,bc)\in \top\,\, \text{and}\,\,(ab)c=a(bc).
\mlabel{eq:wcomp}
\end{equation}
\mlabel{defn:wlocsg}
\end{defn}

\begin{lem}
Let $(\calp,\mtop_\calp,\mu_\calp)$ be the path locality semigroup  defined as in Proposition~\mref{prop:pathsg}. Then $(\calp,\mtop_\calp,\mu_\calp)$ is a \weak locality semigroup.
\mlabel{lem:PathWeak}
\end{lem}
\begin{proof}
Let $(p_1,p_2), (p_2,p_3)\in \mtop_\calp$. Then $t(p_1)=s(p_2)$ and $t(p_2)=s(p_3)$, and so $t(p_1p_2)=t(p_2)=s(p_3)$ and $t(p_1)=s(p_2)=s(p_2p_3)$. This means that $(p_1p_2,p_3) \in \mtop_\calp$ and $(p_1,p_2p_3)\in \mtop_\calp$, and thus
$(p_1p_2)p_3=p_1(p_2p_3).$
\end{proof}
Furthermore, we know that
\begin{prop}Every \weak locality semigroup is a locality semigroup.
\mlabel{prop:weak-loca}
\end{prop}
\begin{proof}
Suppose that $(S,\mtop_S,\mu_S)$ is a \weak locality semigroup. We first prove Eqs.~(\mref{eq:leftclosed}) and (\mref{eq:rightclosed}) in Definition~\mref{defn:locsg}. Let $U\subseteq S$ and let $(a,b)\in (\lt\times\lt)\cap \mtop_S$. Then $(b,u)\in \mtop_S$ for all $u\in U$. Since $(a,b)\in \mtop_S$, we have $(ab,u)\in \mtop_S$ by Eq.~(\mref{eq:wcomp}). This gives $ab\in \lt$, and thus $\mu_S((\lt\times\lt)\cap \mtop_S)\subseteq \lt$. Let $(a,b)\in (U^\top\times U^\top)\cap \mtop_S$. Then $(u,a)\in \mtop_S$ for all $u\in U$. By $(a,b)\in \mtop_S$ and Eq.~(\mref{eq:wcomp}) again, we get $(u,ab)\in \mtop_S$, and so $\mu_S((U^\top\times U^\top)\cap \mtop_S)\subseteq U^\top$. The locality associative law follows from that for every $(a,b),(b,c)\in\mtop_S$, $(ab,c),(a,bc)\in\mtop_S$ and  $(ab)c=a(bc)$.
\end{proof}
But it is not true that every locality semigroup is  a \weak locality semigroup.
\begin{coex}
According to Example~\mref{exam:PosIn},  $(\N^+,\mtop_{cop})$ is a locality semigroup. But it is not a \weak locality semigroup. In fact, we note that $(2,3)\in \mtop_{cop}$ and $(3,4)\in \mtop_{cop}$, but $(6,4)\notin \mtop_{cop}$ and $(2,12)\notin \mtop_{cop}$. This means that $(\N^+,\mtop_{cop})$ does not satisfy Eq.~(\mref{eq:wcomp}).
\mlabel{coex:locsg-weaksg}
\end{coex}
\subsection{Partial semigroups}
In this section, we give the definition of partial semigroups~\mcite{Sch}.

\begin{defn}
\begin{enumerate}
\item
Let $(S,\top)$ be a locality set. A {\bf partial semigroup} is a locality set $(S,\top)$ together with a partial binary operation defined on $S$:
$$\mu_S:S\times_\top S \to S,\quad (a,b)\mapsto ab\,\,\text{for all}\,\, (a,b)\in \top,$$
such that the {\bf partial associative law} holds: for all $a,b,c\in S$, if $(a,b), (b,c)\in \top $, then
\begin{equation}
(ab,c)\in \top \Leftrightarrow (a,bc)\in \top
\mlabel{eq:equivalent}
\end{equation}
and in that case,
\begin{equation}
(ab)c=a(bc).
\mlabel{eq:parasslaw}
\end{equation}
\item
A partial semigroup  $(S,\top)$ is said to be {\bf transitive} if $\top$ is transitive, that is,
$$(a,b)\in \top \,\,\text{and}\,\, (b,c)\in \top \Rightarrow (a,c)\in \top.$$
\end{enumerate}
\mlabel{defn:parsg}
\end{defn}

Firstly, we see that every semigroup is a partial semigroup. Next we give another example of  patrial semigroups.
\begin{exam} Let $S:=\{0,1\}$ and let $\top:=\{(0,0),(0,1),(1,0)\}$. A {\bf partial addition} $+:S\times_\top S\to S$ is defined by
$(0,0)\mapsto 0,\quad(0,1)\mapsto 1,\quad (1,0)\mapsto 1.$  We shall prove that $(S,\top,+)$ is a partial semigroup. There are also five cases: $(0,0),(0,0)\in \top$; $(0,0),(0,1)\in \top$;  $(0,1),(1,0)\in \top$;  $(1,0),(0,0)\in \top$ and $(1,0),(0,1)\in \top$, having the property: $(a,b),(b,c)\in\top$ in Definition~\mref{defn:parsg}. We can check that Eq.~(\mref{eq:equivalent}) in Definition~\mref{defn:parsg} holds for the first four cases. The last case gives $(1+0,1)=(1,1)\notin \top$ and $(1,0+1)=(1,1)\notin\top$, and so Eq.~(\mref{eq:equivalent}) still holds. This means  that $(S,\top,+)$ is a partial semigroup. But it is not a \weak locality semigroup, since for $(1,0)\in \top$ and $(0,1)\in \top$, $(1+0,1)=(1,1)\notin \top$ and $(1,0+1)=(1,1)\notin\top$, that is, $(S,\top,+)$ does not satisfy Eq.~(\mref{eq:wcomp}) in Definition~\mref{defn:wlocsg}.
\mlabel{exam:psnsl}
\end{exam}
Let {\bf PSg} be the category of partial semigroups. Let {\bf LSg} be the category of locality semigroups, and let {\bf SLSg} be the category of \strong locality semigroups. Then by Example~\mref{exam:psnsl}, we have
\begin{equation}
{\bf SLSg} \subsetneq {\bf PSg}.
\mlabel{eq:slps}
\end{equation}
By Proposition~\mref{prop:weak-loca} and Counterexample~\mref{coex:locsg-weaksg},
we obtain
\begin{equation}
{\bf SLSg} \subsetneq {\bf LSg}.
\mlabel{eq:slls}
\end{equation}
Together with Eqs.~(\mref{eq:slps}) and (\mref{eq:slls}),
we get
$${\bf SLSg}\subseteq {\bf LSg}\cap {\bf PSg}.$$
Furthermore,  we show that ${\bf SLSg}$ is a proper subclass of ${\bf LSg}\cap {\bf PSg}$.
\begin{prop}With notations as above, we have
$${\bf SLSg}\subsetneq {\bf LSg}\cap {\bf PSg}.$$
\end{prop}
\begin{proof}
According to Example~\mref{exam:PosIn}, we see that $(\N^+,\mtop_{cop}) $ is a locality semigroup. But it is not a \weak locality semigroup by Counterexample~\mref{coex:locsg-weaksg}.
Next we prove that $(\N^+,\mtop_{cop}) $ is indeed a partial semigroup.
For $a,b,c\in\N^+$, we  let $(a,b)\in \mtop_{cop}$ and $(b,c)\in \mtop_{cop}$. We shall prove that $(ab,c)\in \mtop_{cop}$ if and only if $(a,bc)\in \mtop_{cop}$. Suppose first that $(ab,c)\in \mtop_{cop}$. Then there exist $s,t\in\Z$ such that
\begin{equation}abs+ct=1
\mlabel{eq:coprm1}
\end{equation}
By $(a,b)\in \mtop_{cop}$, there exist $m,n\in\Z$ such that $am+bn=1$, and then multiplying by $ct$, we obtain $amct+bcnt=ct$. By Eq.~(\mref{eq:coprm1}), we get $a(mct+bs)+bcnt=1$. Thus $(a,bc)\in \mtop_{cop}$. Similarly, if $(a,bc)\in \mtop_{cop}$, we can prove $(ab,c)\in\mtop_{cop}$ by using $(b,c)\in\mtop_{cop}$.
This means that $(\N^+,\mtop_{cop})$ satisfies Eq.~(\mref{eq:equivalent}), and so it is  a partial semigroup.  Thus
$(\N^+,\mtop_{cop})\in {\bf LSg}\cap {\bf PSg}$, and $(\N^+,\mtop_{cop})\notin  {\bf SLSg}$.
Thus,
\begin{equation}
{\bf SLSg}\subsetneq {\bf LSg}\cap {\bf PSg}.
\end{equation}
\end{proof}
We now explore the relationships between {\bf LSg} and {\bf PSg}.  For example,
\begin{exam}
Let $S:=\{0,1\}$ and let $\top:=\{(0,0),(0,1),(1,0)\}$.  A partial binary operation $\cdot:S\times_\top S\to S$ is given by
$(0,0)\mapsto 0,\quad(0,1)\mapsto 0,\quad (1,0)\mapsto 1.$ Consider the case: $(1,0),(0,1)\in\top$. Then $(1\cdot 0,1)=(1,1)\notin \top$. But $(1,0\cdot 1)=(1,0)\in\top$. This means that $(S,\top,\cdot)$ does not satisfy Eq.~(\mref{eq:equivalent}) in Definition~\mref{defn:parsg}. Thus $(S,\top,\cdot)$ is not a partial semigroup.

Next we prove that $(S,\top,\cdot)$ is a locality semigroup. Firstly, we verify that Eqs.~(\mref{eq:leftclosed}) and (\mref{eq:rightclosed}) hold for all subset $U$ of $S$.
Three cases arise: $U=\{0\}, U=\{1\}$ and $U=\{0,1\}$. In the first case,  $\lt=\{0,1\}=U^\top$, $(\lt\times \lt )\cap \top= \top$ and $0\cdot0=0\cdot 1=0\in\lt, 1\cdot 0=1\in \lt$.   In the second case,  $\lt=\{0\}=U^\top$, $(\lt\times \lt) \cap \top=\{(0,0)\}$ and $0\cdot 0=0\in\lt$. In the third case,  $\lt=\{0\}=U^\top$, and this is similar to the second case. Thus, Eqs.~(\mref{eq:leftclosed}) and (\mref{eq:rightclosed}) hold.

Secondly, there are four cases: $$(0,0),(0,0),(0,0)\in\top,\,(0,0),(0,1),(0,1)\in\top,$$
$$(0,1),(1,0),(0,0)\in\top, (1,0),(0,0),(1,0)\in \top, $$ satisfying the hypothesis: $(a,b),(b,c),(a,c)\in\top$ in Eq.~(\mref{eq:locass}) of Definition~\mref{defn:locsg}. Note that the case: $(1,0),(0,1)$, does not satisfy the hypothesis since $(1,1)\notin \top$.  We finally check that the locality associative law holds for all cases. By the equations $0\cdot 0=0$, $0\cdot 1=0$ and $1\cdot  0=1$, the first two cases give $(0\cdot 0)\cdot 0=0=0\cdot (0\cdot 0)$ and $(0\cdot 0)\cdot 1=0=0\cdot (1\cdot 0)$. In the third case,  $(0\cdot 1) \cdot 0=0=0\cdot (1\cdot  0)$. In the fourth case,  $(1\cdot 0)\cdot 0=1=1\cdot (0\cdot 0)$.
\mlabel{exam:lsg-psg}
\end{exam}
By Example~\mref{exam:lsg-psg}, we get
\begin{equation}
{\bf LSg} \nsubseteq {\bf PSg}.
\end{equation}
However, if a locality semigroup is transitive, then it is a partial semigroup.

\begin{prop}
Let $(S,\top)$ be a locality semigroup. If $(S,\top)$ is transitive, then $(S,\top)$ is a partial semigroup.
\end{prop}
\begin{proof}
For all $a,b,c\in S$, if $(a,b),(b,c)\in \top$, then $(a,c)\in \top$ by transitivity of $\top$, and so $a,b\in \lts\{c\}$ and $b,c\in\{a\}^\top$. By Eqs.~(\mref{eq:leftclosed}) and (\mref{eq:rightclosed}), $(ab,c)\in \top$ and $(a,bc)\in \top$. Thus Eq.~(\mref{eq:equivalent}) holds, and $(ab)c=a(bc)$ by the locality associative law.
\end{proof}
 We finally verify that {\bf LSg} also does not include {\bf PSg}, that is,
 \begin{equation}
{\bf PSg} \nsubseteq{\bf LSg}.
\end{equation}
In order to prove the above statement, we take an example as follows.
\begin{exam}
Let $S:=\{a,b\}$ be a set. Let $\top:=S\times_\top S:=\{(a,a),(b,b)\}\subseteq S\times S$. Define a partial binary operation $\cdot:S\times_\top S \to S$ by defining
$$(a,a)\mapsto a,  (b,b)\mapsto a.$$
We first verify that $(S,\top)$ is a  partial  semigroup, that is, for every $(x,y),(y,z)\in \top$, we have $(xy,z)$ if and only if $(x,yz)\in\top$, and in that case $(xy)z=x(yz)$. Consider all possible cases: $(a,a),(a,a)\in \top$ and $(b,b),(b,b)\in \top$, having the property that $(x,y),(y,z)\in \top$.  The first case gives $(aa,a)=(a,a)=(a,aa)$ and $(aa)a=a=a(aa)$. In the second case, $(bb,b)=(a,b)\notin\top $ and $(b,bb)=(b,a)\notin\top$, as desired.

We now prove that $(S,\top)$ is not a locality semigroup. Take $U=\{b\}$. Then $\lt=\{b\}$ and $(\lt\times\lt) \cap\top=\{(b,b)\}$. By the equation $bb=a$ and $a\notin \lt$, we know that $(S,\top)$ does not satisfy Eq.~(\mref{eq:leftclosed}) in Definition~\mref{defn:locsg}. This means that $(S,\top,\cdot) \in {\bf PSG}$, but $(S,\top,\cdot)\notin {\bf LSg}$.
\end{exam}
Let {\bf RLSg} be the category of \fine locality semigroups.  By the definition of \fine locality semigroups, every \fine locality semigroup is a \weak locality semigroup, but not vice versa. See Example~\mref{exam:str-fine} below.  Thus, ${\bf RLSg}\subsetneq {\bf SLSg}$.
Now let us put all the pieces together to give the main result of
this section.
\begin{theorem} With notations as above, we have
\begin{equation*}
{\bf RLSg}\subsetneq {\bf SLSg}\subsetneq {\bf LSg}\cap {\bf PSg}\quad\text{and}\quad{\bf LSg}\nsubseteq{\bf PSg}\nsubseteq {\bf LSg}.
\end{equation*}
\mlabel{thm:relation}
\end{theorem}


\section{Constructions of semigroups from \fine locality semigroups}
\mlabel{sec:stro}
In this section, we first describe how to obtain a partial semigroup from any given full semigroup.
Let $S:=(S,\mu_S)$ be a semigroup. We can get a partial semigroup from $S$ by the following way. Let $A$ be a nonempty set of $S$. Denote
$$\mtop_A:=A\times_\top A:=\{(a,b)\in A\times A\,|\,a b\in A\}.$$
Let $\mu_S|_{\mtop_A}$ be the restriction of $\mu_S$ to $\mtop_A$.  Then we  get a partial binary operation $\mu_S|_{\mtop_A}:A\times_\top A\to A,\,(a,b)\mapsto ab $. Thus,  $(A,\mtop_A,\mu_S|_{\mtop_A})$ is a partial semigroup, since $(ab,c)\in \mtop_A$ if and only if $(a,bc)\in\mtop_A$ for $(a,b),(b,c)\in\mtop_A$. In particular, if $A$ is a subsemigroup of $S$, then $\mtop_A=A\times A$. This means that every subsemigroup is a partial semigroup.

Now a natural question arises:

\vspace{4pt}
{\bf Question}\, How to construct a semigroup from any given partial semigroup?
\vspace{4pt}

By Theorem~\mref{thm:relation}, we see that every \fine locality semigroup is a special partial semigroup. Moreover,
we can construct a full semigroup from any  \fine locality semigroup as follows.
\begin{theorem}
Let $(S,\mtop_S,\cdot)$ be a \fine locality semigroup. Let $0\notin S$ and let $S^0:=S\cup \{0\}$.  Define a binary operation $\ast:S^0\times S^0\to S^0$ by defining
\begin{equation*}
(x,y)\mapsto x\ast y:=\left\{\begin{array}{lll}
x \cdot y, &(x,y)\in \mtop_S;\\
0,&(x,y)\notin \mtop_S.\\
\end{array}\right.
\end{equation*}
Then $(S^0,\ast)$ is a semigroup with zero.
\mlabel{thm:sgzero}
\end{theorem}
\begin{proof}
We shall prove  the associative law: for all $x,y,z\in S^0$, $(x\ast y)\ast z=x\ast (y\ast z)$.
Two cases arise.
\smallskip

\noindent
{\bf Case 1.} $(x,y)\in \mtop_S$.  Then there are two subcases depending on whether or not $(y,z)\in \mtop_S $.
If $(y,z)\in \mtop_S$, then $(x\cdot y,z),(x,y\cdot z)\in\mtop_S$ by  Eq.~(\mref{eq:refinedremark}), and thus $(x\cdot y)\cdot z=x\cdot (y\cdot z)$ by Definition~\mref{defn:flocsg}(\mref{it:fineass}). By the definition of $\ast$, we obtain
$$(x\ast y)\ast z=(x\cdot y)\cdot z=x\cdot(y\cdot z)=x\ast (y\ast z).$$
If $(y,z)\notin \mtop_S$, then $(x\cdot y,z)\notin \mtop_S$ by Definition~\mref{defn:flocsg}(\mref{it:fineone}), and so $(x\ast y)\ast z=0$. Hence,
$$(x\ast y)\ast z=0=x\ast(y\ast z).$$
\smallskip

\noindent
{\bf Case 2.} $(x,y)\notin \mtop_S$. Then $x\ast y=0$, and so $(x\ast y)\ast z=0$. In order to prove the associative law, we also consider two subcases: $(y,z)\in \mtop_S $  and $(y,z)\notin \mtop_S$.
If $(y,z)\in\mtop_S$, then $(x,y\cdot z)\notin \mtop_S$ by Definition~\mref{defn:flocsg}(\mref{it:finetwo}). This gives
$$x\ast (y\ast z)=0=(x\ast y)\ast z.$$
If $(y,z)\notin \mtop_S$, then $y\ast z=0$, and thus
$$x\ast (y\ast z)=0=(x\ast y)\ast z.$$
Then $(x\ast y)\ast z=x\ast (y\ast  z)$ and $x\ast0=0 \ast x=0$ for all $x,y,z\in S^0$. It follows that $(S^0,\ast)$ is a semigroup with zero.
\end{proof}

\begin{remark}
For any given  partial semigroup $(S,\cdot,\mtop_S)$, it is not necessary  to  construct a semigroup by adding elements, such as, zero element, to it as follows.
Firstly, let $0\notin S$ and let $S^0:=S\cup\{0\}$. Then we define a binary operation  $\mu_{S^0}:S^0\times S^0\to S^0$  by defining
$$(x,y)\mapsto\mu_{S^0}(x,y):=\left\{\begin{array}{lll}
x\cdot y, &(x,y)\in \mtop_S;\\
0,&(x,y)\notin \mtop_S.\\
\end{array}\right.$$
But  $(S^0,\mu_{S^0})$ may not be a semigroup.
\end{remark}
The next example shows that it really not necessary to extend any partial semigroup $(S,\cdot,\mtop_S)$ to a full semigroup using the above method even if $(S,\cdot,\mtop_S)$ is a strong locality semigroup.
\begin{exam}
Let $S:=\{a,b\}$ be a set. Let $\top:=S\times_\top S:=\{(a,a),(a,b)\}\subseteq S\times S$. Define a partial binary operation $\cdot:S\times_\top S \to S$ by defining
$$(a,a)\mapsto a, (a,b)\mapsto a.$$
We shall  show that $(S,\top)$ is a strong locality semigroup, that is, for every $(x,y),(y,z)\in \top$, we have $(xy,z),(x,yz)\in\top$, and $(xy)z=x(yz)$. Consider all possible cases: $(a,a),(a,a)\in \top$ and $(a,a),(a,b)\in \top$.  In the first case,  since $aa=a$, we have $(aa,a),(a,aa)\in\top$, and hence $(aa)a=a(aa)$.
 In the second case, we obtain $(aa,b)=(a,b)\in\top$, $(a,ab)=(a,a)\in\top$ and $(aa)b=a=a(ab)$. It follows that $(S,\top,\cdot)$ is a strong locality semigroup.

Secondly,  we verify that  $(S,\top,\cdot)$ is not a \fine locality semigroup.
Since $(a,b)\in\top$ and $ab=a$, we have $(ab,a)=(a,a)\in\top$. But $(b,a)\notin\top$, and hence $(S,\top,\cdot)$ does not satisfy Definition~\mref{defn:flocsg}(\mref{it:fineone}) (that is, if $(x,y)\in\top$, then $(y,z)\in \top$ if and only if $(xy,z)\in\top$). Thus,  $(S,\top,\cdot)$ is not a \fine locality semigroup.

Finally, let $0\notin S$ and let $S^0:=S\cup\{0\}$.
Define a binary operation  $\mu_{S^0}:S^0\times S^0\to S^0$  by defining
$$(x,y)\mapsto\mu_{S^0}(x,y):=\left\{\begin{array}{lll}
x\cdot y, &(x,y)\in \top;\\
0,&(x,y)\notin \top.\\
\end{array}\right.$$
Since $(ab)a=a\neq 0=a(ba)$,  $(S^0,\mu_{S^0})$ is not a semigroup.
\mlabel{exam:str-fine}
\end{exam}

According to Theorem~\mref{thm:sgzero}, we have
\begin{prop}Let $(S,\mtop_S,\cdot)$ be a \fine locality semigroup. Let $(S^0,\ast)$ be the semigroup with zero defined in Theorem~\mref{thm:sgzero}. Then
\begin{equation}
xyz\neq0\Leftrightarrow xy\neq 0\,\,\text{and}\,\, yz\neq 0\quad\text{for all}\,\, x,y,z\in S^0.
\end{equation}
\mlabel{prop:stgpro}
\end{prop}
\begin{proof}$(\Rightarrow)$ is trivial. \\
$(\Leftarrow)$ By the hypothesis that $xy\neq 0$ and $yz\neq 0$, we get $(x,y),(y,z)\in\mtop_S$. Since $(S,\mtop_S,\cdot)$ is  a \fine locality semigroup, we have $(xy,z),(x,yz)\in\mtop_S$. This gives $(xy)z,x(yz)\in S$, and thus
 $$xyz=(xy)z=x(yz)\neq 0.$$
\end{proof}

\begin{defn}Let $S^0$ be a semigroup with zero. We say that $S^0$ is {\bf strong} if
\begin{equation}
abc\neq0\Leftrightarrow ab\neq 0\,\,\text{and}\,\, bc\neq 0\quad\text{for all}\,\, a,b,c\in S^0.
\mlabel{eq:xyz}
\end{equation}
\mlabel{defn:strsgzero}
\end{defn}
Note that Eq.~(\mref{eq:xyz}) is equivalent to
\begin{equation}
abc=0\Leftrightarrow ab= 0\,\,\text{or}\,\, bc= 0\quad\text{for all}\,\, a,b,c\in S^0.
\end{equation}
\begin{prop}Let $(\calp,\mtop_\calp,\mu_\calp)$ be the path locality semigroup defined as in Proposition~\mref{prop:pathsg}. Let $0\notin \calp$ and let $\calp^0:=\calp\cup \{0\}$. Define a binary operation $\ast_\calp:\calp^0\times \calp^0\to \calp^0$ by defining
\begin{equation*}
(p,q)\mapsto p\ast_\calp q:=\left\{\begin{array}{lll}
pq, &(p,q)\in \mtop_\calp;\\
0,&(p,q)\notin \mtop_\calp.\\
\end{array}\right.
\end{equation*}
Then $(\calp^0,\ast_\calp)$ is a strong semigroup with zero, called a {\bf path semigroup with zero}, or simply a {\bf path semigroup}.
\mlabel{prop:pathsgzero}
\end{prop}
\begin{proof}
By Proposition~\mref{prop:pathresg}, we know that the path locality semigroup $(\calp,\mtop_\calp,\mu_\calp)$ is a \fine locality semigroup. Then by Theorem~\mref{thm:sgzero} and Proposition~\mref{prop:stgpro}, $(\calp^0,\ast_\calp)$ is a strong semigroup with zero.
\end{proof}
\smallskip

\noindent
{\bf Acknowledgements.} This work was supported by the National Natural Science Foundation of
China (Grant No.~11601199)  and the Foundation of Jiangxi Provincial Education Department (Grant No. GJJ160336).

\end{document}